\documentclass[11pt]{article}
\usepackage{epigamath}

\usepackage[notext]{kpfonts}
\usepackage{baskervald}

\setpapertype{A4}

\usepackage[english]{babel}

\usepackage[dvips]{graphicx}     

\title{Socle pairings on tautological rings}
\titlemark{Socle pairings on tautological rings}
\author{\vspace{0cm} Felix Janda and Aaron Pixton}
\authoraddresses{
\authordata{Felix Janda}{\firstname{Felix} \lastname{Janda}\\
\institution{Department of Mathematics, University of Michigan, 2074 East Hall, 530 Church Street, Ann Arbor, MI 48109, USA}\\
\email{janda@umich.edu}}\\
\authordata{Pixton}{\firstname{Aaron} \lastname{Pixton}\\
\institution{Department of Mathematics, Massachusetts Institute of Technology, Cambridge, MA 02139, USA}\\
\email{apixton@mit.edu}}
}
\authormark{F. Janda and A. Pixton}
\date{\vspace{-5ex}} 
\journal{\'Epijournal de G\'eom\'etrie Alg\'ebrique} 
\acceptation{Received by the Editors on July 11, 2017, and in final form
on February 15, 2018. \\ Accepted on February 9, 2019.}

\newcommand{\articlereference}{Volume 3 (2019), Article Nr. 4}

\usepackage[all]{xy}

\allowdisplaybreaks

\usepackage{tikz}

\newtheorem{lem}{Lemma}
\newtheorem{thm}{Theorem}

\begin{document}


\maketitle



\begin{prelims}


\def\abstractname{Abstract}
\abstract{ We study some aspects of the $\lambda_g$-pairing on the tautological
  ring of $M_g^c$, the moduli space of genus $g$ stable curves of
  compact type. We consider pairing $\kappa$-classes with pure
  boundary strata, all tautological classes supported on the boundary,
  or the full tautological ring. We prove that the rank of this
  restricted pairing is equal in the first two cases and has an
  explicit formula in terms of partitions, while in the last case the
  rank increases by precisely the rank of the $\lambda_g\lambda_{g -
    1}$-pairing on the tautological ring of $M_g$.}



\vspace{0.15cm}

\languagesection{Fran\c{c}ais}{%

\textbf{Titre. Accouplements-socles sur les anneaux tautologiques} \commentskip \textbf{R\'esum\'e.} Nous \'etudions certains aspects de l'accouplement $\lambda_g$ sur l'anneau tautologique de $M_g^c$, l'espace de modules des courbes stables de type compact de genre $g$. Nous consid\'erons l'accouplement de classes $\kappa$ avec des strates pures du bord, avec toutes les classes tautologiques support\'ees sur le bord ou bien avec l'anneau tautologique dans sa totalit\'e. Nous montrons que le rang de cet accouplement restreint est le m\^eme dans les deux premiers cas et a une expression explicite en termes de partitions, tandis que dans le dernier cas, le rang augmente pr\'ecis\'ement du rang  de l'accouplement $\lambda_g\lambda_{g - 1}$ sur l'anneau tautologique de $M_g$.}

\end{prelims}


\newpage

\setcounter{tocdepth}{1} \tableofcontents

\setcounter{section}{0}
\section{Introduction}

Let $M_{g, n}$ be the moduli space of smooth curves of genus $g$ with
$n$ marked points and let $\overline M_{g, n}$ be the Deligne--Mumford
compactification, the moduli space of stable $n$-pointed nodal curves
of arithmetic genus $g$. Inside this, let $M_{g, n}^c$ be the subspace
of stable pointed curves of compact type, i.e.\ curves whose dual graph
is a tree.

The intersection theory of these moduli spaces of curves is a subject
of fundamental importance in algebraic geometry. When studying the
Chow ring $A^*(\overline M_{g, n})$, one is naturally led to consider
a subring consisting of the classes such as the Arbarello--Cornalba
$\kappa$-classes that are defined via certain tautological maps between
the $\overline M_{g, n}$. This subring is the {\it tautological ring}
$R^*(\overline M_{g. n})$. Tautological rings $R^*(M_{g, n})$ and
$R^*(M_{g, n}^c)$ for $M_{g, n}$ and $M^c_{g, n}$ can be defined by
restriction. We will primarily be interested in $R^*(M_g^c)$, the case
of compact type with no marked points.

Inside $R^*(M_{g, n}^c)$ there is the subring $\kappa^*(M_{g, n}^c)$
generated by the $\kappa$-classes $\kappa_1,\kappa_2,\ldots$. The kappa
ring $\kappa^*(M_{g, n}^c)$ has been studied in detail by
Pandharipande \cite{Pandharipande}. In particular, for $n > 0$ a
complete description of the kappa ring is given. For this reason we
concentrate on the case $n = 0$ in this paper.

When restricted to the moduli space of smooth curves $M_g$, the
tautological ring $R^*(M_g)$ is actually equal to the kappa ring
$\kappa^*(M_g)$. This means that on $M_g^c$, any tautological class
can be written as the sum of a polynomial in the $\kappa$-classes and
a class supported on the boundary. We denote by $BR^*(M_g^c)$ the
ideal of tautological classes supported on the boundary, so the
tautological ring $R^*(M_g^c)$ is linearly spanned by
$\kappa^*(M_g^c)$ and $BR^*(M_g^c)$.

A general element of $BR^*(M_g^c)$ is a linear combination of classes
obtained by taking the pushforward of tautological classes via gluing
maps
\begin{equation*}
  M_{g_1,n_1}^c\times M_{g_2,n_2}^c\times \cdots \times M_{g_k,n_k}^c \to M_g^c.
\end{equation*}
When the class $1$ is pushed forward along such a map, this
construction gives a pure boundary stratum. We let $PBR^*(M_g^c)$ denote
the linear subspace of $BR^*(M_g^c)$ generated by the pure boundary
strata.

There are natural bilinear pairings
\begin{align*}
  R^r(M_g^c) \times R^{2g - 3 - r}(M_g^c) \to R^{2g - 3}(M_g^c) \cong \mathbb Q, \\
  R^r(M_g) \times R^{g - 2 - r}(M_g) \to R^{g - 2}(M_g) \cong \mathbb Q,
\end{align*}
given by the product in the Chow ring and the socle evaluations.
These pairings are called the $\lambda_g$- and
$\lambda_g \lambda_{g - 1}$-pairings respectively because they may be
defined by integrating against these classes in $\overline M_g$.

In this paper we will study the restriction
\begin{equation*}
  \kappa^d(M_g^c) \times R^r(M_g^c) \to \mathbb Q
\end{equation*}
of the $\lambda_g$-pairing for $r + d = 2g - 3$, for any $g\ge 2$. The following
theorems, our main results, were previously conjectured by
Pandharipande.

\bigskip
\noindent
{\bf Housing Theorem.}
{\em
  The rank of the $\lambda_g$-pairing of $\kappa$-classes against
  boundary classes
  \begin{equation*}
    \kappa^d(M_g^c) \times BR^r(M_g^c) \to \mathbb Q
  \end{equation*}
  equals the rank of the $\lambda_g$-pairing of $\kappa$-classes
  against pure boundary strata
  \begin{equation*}
    \kappa^d(M_g^c) \times PBR^r(M_g^c) \to \mathbb Q.
  \end{equation*}
  Furthermore, these ranks are equal to the number of partitions of
  $d$ of length less than $r + 1$ plus the number of partitions of $d$
  of length $r + 1$ which contain at least two even parts.}
  
\bigskip
\noindent
{\bf Rank Theorem.}
{\em
  The rank of the $\lambda_g$-pairing of $\kappa$-classes against
  general tautological classes
  \begin{equation*}
    \kappa^d(M_g^c) \times R^r(M_g^c) \to \mathbb Q
  \end{equation*}
  equals the sum of the rank of the $\lambda_g$-pairing of
  $\kappa$-classes against boundary classes
  \begin{equation*}
    \kappa^d(M_g^c) \times BR^r(M_g^c) \to \mathbb Q
  \end{equation*}
  and the rank of the $\lambda_g\lambda_{g - 1}$-pairing
  \begin{equation*}
    \kappa^r(M_g) \times \kappa^{g - 2 - r}(M_g) \to \mathbb Q.
  \end{equation*}}

These theorems will be proven by direct combinatorial analysis of the
well~known formulae for calculating the integrals arising in the
pairings. In particular, we have no geometric explanation of the
Rank~Theorem, which connects the compact-type case and the smooth
case.

\subsection{Consequences}\label{sec:consq}

It has been conjectured by Faber \cite{MR1722541} that $\kappa^*(M_g)
= R^*(M_g)$ is a Gorenstein ring with socle in degree $g - 2$. He
verified this for $g \le 23$ by computing many relations between the
$\kappa$-classes and checking that they produced a Gorenstein
ring. However, starting in genus $24$, the known methods of producing
relations have failed to give enough relations to yield a Gorenstein
ring. In fact, the known relations have all been in the span of the
Faber--Zagier (FZ) relations, and these relations produce a Gorenstein
ring if and only if $g \le 23$.

There are therefore \emph{mystery relations} in $R^*(M_g)$: formal
polynomials in $\kappa$-classes which pair to zero with any $\kappa$
polynomial in $R^*(M_g)$ of complementary degree but are not a linear
combination of FZ relations. If one assumes Faber's Gorenstein
conjecture then these relations must hold in $R^*(M_g)$. Since FZ
relations extend to tautological relations in $R^*(\overline M_g)$
(this is a consequence of the proof of the FZ relations in
\cite{Pandharipande-Pixton}), a possible reason for the existence of
mystery relations might be if they do not extend tautologically to
$R^*(M_g^c)$ or $R^*(\overline M_g)$. The Rank Theorem can be
interpreted as saying that part of the obstruction to this extension
is zero: the mystery relations at least extend to classes in the
tautological ring of $M_g^c$ which pair to zero with the $\kappa$
subring. It is an interesting question whether the mystery relations
extend to classes in the tautological ring of $M_g^c$ which are
relations in the Gorenstein quotient (i.e.\ pair to zero with the
entire tautological ring).

In \cite{Pandharipande} Pandharipande gives a minimal set of
generators of $\kappa^*(M_{g, n}^c)$ for $n > 0$ and relates higher genus
relations to genus 0 relations. More precisely, he shows that there is
a surjective (graded) ring homomorphism
\begin{equation*}
  \kappa^*(M_{0, 2g+n}^c) \stackrel{\iota_{g, n}}{\to} \kappa^*(M_{g, n}^c),
\end{equation*}
which is an isomorphism for $n \ge 1$, or in degrees up to $g-2$ when
$n = 0$. The Rank Theorem gives us information about the $n = 0$ case
in higher degrees.
\begin{thm}
  Let $g \ge 2$, $0 \le e \le g-2$, and $d = g-1+e$. Let $\delta_d$ be
  the rank of the kernel of the map from $\kappa^d(M_g^c)$ to the
  Gorenstein quotient of $R^*(M_g^c)$. Let $\gamma_e$ be the rank of
  the space of $\kappa$-relations of degree $e$ in the Gorenstein
  quotient of $R^*(M_g)$. Let $N_e$ denote the number of partitions of $e$ of length greater than $g-1-e$. Then the degree-$d$ part of the kernel of
  $\iota_{g,0}$ has rank $\gamma_e - \delta_d - N_e$.
\end{thm}
\begin{proof}
  We use the notation $|P(m)|$ for the number of partitions of $m$ and $|P(m,k)|$ for the number of partitions of $m$ of length at most $k$, so $N_e = |P(e)| - |P(e,g-1-e)|$. By \cite{Pandharipande}, the rank of $\kappa^d(M_{0, 2g}^c)$ is
  equal to $|P(d, 2g-2-d)|$. On the other side, the rank of $\kappa^d(M_g^c)$
  is equal to $\delta_d$ plus the rank of the first pairing appearing
  in the Rank Theorem. The rank of the second pairing appearing in the
  Rank Theorem is given by the Housing Theorem and is equal to $|P(d, 2g-2-d)| - X$, where $X$ is the number of partitions of $d$ of length $2g-2-d$ with no even parts. By subtracting one from each part and dividing by two, we have that $X = |P(e,g-1-e)|$. The rank of the
  third pairing appearing in the Rank Theorem is equal to
  $|P(e)| - \gamma_e$. Putting all these pieces together gives
\begin{align*}
&\dim_{\mathbb Q}\kappa^d(M_{0, 2g}^c) - \dim_{\mathbb Q}\kappa^d(M_g^c)\\
&= |P(d, 2g-2-d)| - \left(\delta_d + |P(d, 2g-2-d)| - |P(e,g-1-e)| + |P(e)| - \gamma_e\right)\\
&= \gamma_e - \delta_d - N_e,
\end{align*}
as desired. \hfill $\Box$
\end{proof}

\noindent{\bf Remark.}
The components $\gamma_e$ and $\delta_d$ appearing in the above
theorem both have conjectural values. The FZ relations give a
prediction for $\gamma_e$ (if they are the only relations in the first
half of the Gorenstein quotient and are linearly independent):
\[
\gamma_e = \begin{cases}
a(3e-g-1) &\text{ if }e \le \frac{g-2}{2}\\
a(3(g-2-e)-g-1) + |P(e)| - |P(g-2-e)| &\text{ else},\end{cases}
\]
where $a(n)$ is the number of partitions of $n$ with no parts of sizes
$5,8,11,14,\ldots$.

Although the Gorenstein conjecture in compact type is false (see
\cite{MR3530445}), it is still reasonable to predict that the kernel of
the map from $R^d(M^c_g)$ to its Gorenstein quotient will fail to
intersect the relatively small subring $\kappa^d(M^c_g)$. This would
imply that $\delta_d = 0$. Combining this prediction with the FZ
prediction for $\gamma_e$ gives a conjecture for all the Betti numbers
of $\kappa^*(M^c_g)$: we expect that
\begin{equation*}
\dim_{\mathbb Q}\kappa^{d}(M_g^c) = |P(d,2g-2-d)| - a(3d-4g+2) \quad \text{ if } 0\le d\le \frac{3g-2}{2},
\end{equation*}
along with a slightly more complicated formula in the case $\frac{3g-2}{2} < d \le 2g-3$.


\subsection{Plan of the paper}

In Section~\ref{sec:defs}, we review basic facts about the
tautological ring. 
In Section~\ref{sec:housing}, we prove the Housing Theorem.
Finally, in Section~\ref{sec:rank}, we state and prove a slightly more
explicit version of the Rank Theorem (see Theorem~\ref{thm:rank}).

\subsection*{Acknowledgments}

The first named author wants to thank his advisor Rahul Pandharipande
for the introduction to this topic and various discussions. The beginning of Section
\ref{sec:consq} elaborates an email from him.

\
\hspace{-0.25cm}
The\,first\,named\,author\,was\,supported\,by\,the\,Swiss\,National\,Science\,Foundation\,grant\,SNF\,200021\_143274. The second named author was
supported by an NDSEG graduate fellowship.

\section{The tautological ring}\label{sec:defs}

\subsection{Tautological Classes}\label{sec:generators}

The subrings $R^*(\overline M_{g. n})$ of tautological classes in the
Chow rings $A^*(\overline M_{g, n})$ are collectively defined as the
smallest subrings which are closed under pushforward via the maps
forgetting markings $\overline M_{g, n} \to \overline M_{g, n - 1}$
and the gluing maps
\begin{equation*}
  \overline M_{g_1, n_1 \sqcup \{\star\}} \times \overline M_{g_2, n_2 \sqcup \{\bullet\}} \to \overline M_{g_1 + g_2, n_1 + n_2}
\end{equation*}
\begin{equation*}
  \text{and}
\end{equation*}
\begin{equation*}
  \overline M_{g, n \sqcup \{\star, \bullet\}} \to \overline M_{g + 1, n}
\end{equation*}
defined by gluing together $\star$ and $\bullet$.  It turns out that
nearly all classes on the moduli space of curves that appear naturally
in geometry lie in the tautological ring.

For each $i = 1,2,\ldots,n$, there is a line bundle $\mathbb{L}_i$ on
$\overline{M}_{g,n}$ given by the cotangent space at the $i$th marked
point. The first Chern classes of these line bundles are denoted by $\psi_i = c_1(\mathbb{L}_i) \in A^1(\overline{M}_{g,n})$. The
$\kappa$-classes are then pushforwards of powers of the $\psi$
classes:
\[
\kappa_m = \pi_*(\psi_{n+1}^{m+1}) \in A^m(\overline{M}_{g,n}),
\]
where $\pi$ is the forgetful map $\overline M_{g, n + 1} \to \overline M_{g, n}$.

It is well known (see e.g.\ \cite{MR1960923}) that the $\kappa$- and
$\psi$-classes combined with pushforward by the gluing morphisms alone
are sufficient to generate the tautological rings. In other words,
$R^*(\overline{M}_{g,n})$ is additively generated by classes of the
form
\begin{equation*}
  \xi_{\Gamma*} \left(\prod_{v\text{ vertex of }\Gamma} \theta_v\right),
\end{equation*}
where $\Gamma$ is a stable graph expressing the data of the gluing map
\begin{equation*}
  \xi_\Gamma\colon \prod_{v\text{ vertex of }\Gamma} \overline{M}_{g(v), n(v)} \to \overline{M}_{g,n}
\end{equation*}
and the $\theta_v \in R^*(\overline{M}_{g(v), n(v)})$ are arbitrary monomials in the
$\psi$- and $\kappa$-classes.

The tautological rings $R^*(M_{g,n}^c)$ and $R^*(M_{g,n})$ are defined
as the image of $R^*(\overline{M}_{g,n})$ under restriction. In the
case of $R^*(M_{g,n}^c)$, this means that the stable graph $\Gamma$
must be a tree, while $R^*(M_{g,n})$ is simply the subring of
polynomials in the $\kappa$- and $\psi$-classes.

The ring $R^*(M_{g,n}^c)$ has one-dimensional socle, in degree
$2g-3+n$:
\[
R^{2g-3+n}(M_{g,n}^c) \cong \mathbb{Q}.
\]
This gives a canonical (up to scaling) bilinear pairing on
$R^*(M_{g,n}^c)$, which can be realized explicitly by integrating
against the Hodge class $\lambda_g$:
\begin{equation*}
  R^*(M_{g,n}^c) \times R^*(M_{g,n}^c) \to \mathbb Q, \quad (\alpha, \beta) \mapsto \int_{\overline M_{g,n}} \alpha\beta\lambda_g.
\end{equation*}
Here, the integral is defined by taking any extensions of $\alpha$ and
$\beta$ to $R^*(\overline M_{g,n})$. It is independent of which
particular extension one has chosen because $\lambda_g$ vanishes on
the complement of $M_{g,n}^c$.

The $\lambda_g\lambda_{g - 1}$-pairing is a similar pairing for the
moduli space of smooth curves, given by
\begin{equation*}
  R^*(M_g) \times R^*(M_g) \to \mathbb Q, \quad (\alpha, \beta) \mapsto \int_{\overline M_g} \alpha\beta\lambda_g\lambda_{g - 1}.
\end{equation*}
Notice that the $\lambda_g$-pairing on $R^*(M_g^c)$ vanishes above
degree $2g - 3$ whereas the $\lambda_g\lambda_{g - 1}$-pairing on
$R^*(M_g)$ already vanishes above degree $g - 2$.

\subsection{Notation concerning partitions}\label{sec:part}

In the following sections we will use the following notation
heavily. A \emph{partition} $\sigma$ is an unordered collection of
natural numbers (a multiset). We call its elements \emph{parts}. Its
\emph{size} is the sum of all its parts. The \emph{length}
$\ell(\sigma)$ of a partition $\sigma$ is the number of parts in
it. For natural numbers $n$ and $r$, we denote by $P(n)$ the set of
partitions of size $n$ and by $P(n, r)$ the set of partitions of size
$n$ and length at most $r$. Furthermore, let $I(\sigma)$ be a set of
$\ell(\sigma)$ elements which we will use to index the parts of
$\sigma$. For example we could take
\begin{equation*}
  I(\sigma) = [\ell(\sigma)] := \{1, \dots, \ell(\sigma)\}.
\end{equation*}
For two partitions $\sigma, \tau \in P(n)$ and a map $\varphi\colon
I(\sigma) \to I(\tau)$ we say that $\varphi$ is a \emph{refining
  function} of $\tau$ into $\sigma$ if for any $i \in I(\tau)$ we have
\begin{equation*}
  \tau_i = \sum_{j \in \varphi^{-1}(i)}\sigma_j.
\end{equation*}
If for given $\sigma$ and $\tau$ there exists a refining function
$\varphi$ of $\tau$ into $\sigma$, we say that $\sigma$ is a
refinement of $\tau$.

For a finite set $S$, a \emph{set~partition} $P$ of $S$ (written $P
\vdash S$) is a set $P = \{S_1, \dots, S_m\}$ of nonempty subsets of
$S$ such that $S$ is the disjoint union of the $S_i$.

For a partition $\sigma$ and a set $S$ of subsets of $I(\sigma)$ we
define a new partition $\sigma^S$ indexed by the elements of $S$ by
setting $(\sigma^S)_s = \sum_{i \in s} \sigma_i$ for each $s \in
S$. Usually we will take a set partition $P$ of $I(\sigma)$ for
$S$. For a subset $T \subseteq I(\sigma)$, we define the
\emph{restriction $\sigma|_T$ of $\sigma$ to $T$} by $\sigma^S$, where
$S$ is the set of all 1-element subsets of $T$; in other words,
$\sigma|_T = (\sigma_t)_{t \in T}$.

\subsection{Integral calculations}

The basic formula for the evaluation of the integrals arising in the
$\lambda_g$-pairing is (see \cite{MR1954265})
\begin{equation*}
  \int_{\overline M_{g, n}} \prod_{i = 1}^n \psi_i^{\tau_i} \lambda_g = \binom{2g - 3 + n}\tau \int_{\overline M_{g, 1}} \psi_1^{2g - 2} \lambda_g,
\end{equation*}
where $\tau_1, \dotsc, \tau_n$ are nonnegative integer numbers with
sum $2g - 3 + n$. The formula is symmetric with respect to the sorting
of the markings and hence we only need to know the partition
corresponding to $\tau$ in order to calculate these integrals.
Since we are only interested in the ranks of the pairing, the only
thing we will need to know about the integral on the right hand side
is that it is nonzero (see \cite{MR1728879}).

We will need to evaluate integrals involving $\psi$-classes as part of
the proof of the housing theorem. However our main interest lies in
the calculation of integrals involving $\kappa$-classes. Using the
definition of the $\kappa$-classes as push-forwards of powers of
$\psi$-classes, we can find a nice expression for the quotients
\begin{equation*}
  \vartheta(\sigma; \tau) := \left(\int_{\overline M_{g, \ell(\tau)}} \kappa_\sigma\psi^\tau \lambda_g\right) \left(\int_{\overline M_{g, 1}} \psi_1^{2g - 2} \lambda_g\right)^{-1}.
\end{equation*}
In this equation we have used $\kappa_\sigma$ as an abbreviation for
$\prod_{i \in I(\sigma)} \kappa_{\sigma_i}$ and $\psi^\tau$ for
$\prod_{i \in I(\tau)} \psi_i^{\tau_i}$ indexing the $|\tau|$ marked
points by the parts of $\tau$.
We will write
\begin{equation*}
  \vartheta(\sigma) := \vartheta(\sigma; \emptyset)
\end{equation*}
when we just have $\kappa$-classes and no $\psi$-classes.

\begin{lem}\label{lem:kappapsi}
  For partitions $\sigma$ and $\tau$ such that $2g - 3 + \ell(\tau) =
  |\sigma| + |\tau|$, we have
  \begin{equation*}
    \vartheta(\sigma; \tau) = \sum_{P \vdash I(\sigma)} (-1)^{|P| + \ell(\sigma)} \binom{2g - 3 + |P| + \ell(\tau)}{((\sigma^P)_i + 1)_{i \in P}, \tau}.
  \end{equation*}
\end{lem}
\begin{proof}
  From the basic socle evaluation formula we see that it suffices to
  prove the identity
  \begin{equation*}
    \kappa_\sigma \psi^\tau \lambda_g = \sum_{P \vdash I(\sigma)} (-1)^{|P| + \ell(\sigma)} \pi_*\left(\psi^{((\sigma^P)_i + 1)_{i \in P}}\psi^\tau \lambda_g\right).
  \end{equation*}
  in the Chow ring $R(\overline M_{g, \ell(\tau)})$, where by abuse of
  notation $\pi$ is the forgetful map $\overline M_{g, \ell(\tau) + n}
  \to \overline M_{g, \ell(\tau)}$ for the appropriate $n$. Since
  $\pi^*(\lambda_g) = \lambda_g$, we can further reduce to
  \begin{equation*}
    \kappa_\sigma \psi^\tau = \sum_{P \vdash I(\sigma)} (-1)^{|P| + \ell(\sigma)} \pi_*\left(\psi^{((\sigma^P)_i + 1)_{i \in P}}\psi^\tau \right).
  \end{equation*}
  
  This follows from the pushforward formula
  \begin{equation*}
    \pi_*\left(\psi^{(\sigma_i + 1)_{i \in P}}\psi^\tau \right) = \sum_{P \vdash I(\sigma)}\left(\prod_{S\in P}(|S|-1)!\right) \kappa_{\sigma^P}\psi^\tau.
  \end{equation*}
  and partition refinement inversion.
  \hfill $\Box$
\end{proof}

To evaluate the more general integrals which arise when we pair
$\kappa$-classes with arbitrary tautological classes, we can restrict
ourselves to pairing a $\kappa$-monomial with the additive set of
generators described in Section~\ref{sec:generators}. In this case, we
have to sum over the set of possible distributions of the
$\kappa$-classes to the vertices of $\Gamma$ and then multiply the
$\lambda_g$-integrals at each vertex.

The $\lambda_g\lambda_{g - 1}$-pairing formula is similar:
\begin{equation*}
  \int_{\overline M_{g, n}} \psi^\sigma\lambda_g\lambda_{g - 1} = \frac{(2g - 3 + \ell(\sigma))!(2g - 1)!!}{(2g - 1)!\prod_{i \in I(p)}(2\sigma_i + 1)!!}\int_{\overline M_g} \psi^{g - 2}\lambda_g\lambda_{g - 1}.
\end{equation*}
The integral on the right hand side is known to be nonzero (see
\cite{MR1653492}). We can calculate the $\kappa$-integrals analogously
to Lemma~\ref{lem:kappapsi}.

\section{The Housing Theorem}\label{sec:housing}

\subsection{Housing Partitions}\label{sec:housepart}

Let us now study pairing $\kappa$-monomials of degree $d$ with pure
boundary classes via the $\lambda_g$-pairing. Each pure boundary
stratum in codimension $2g - 3 - d$ is determined by a tree $\Gamma =
(V, E)$ with $|V| = 2g - 2 - d$ vertices and $|E| = 2g - 3 - d$ edges,
and a genus function $g\colon V \to \mathbb Z_{\ge 0}$ with $\sum_{v \in V}
g(v) = g$. Then, the class is the push-forward of $1$ along the gluing
map $\xi_\Gamma\colon\prod_{v \in V} M_{g(v), n(v)}^c \to M_g^c$
corresponding to the tree $\Gamma$, where $n(v)$ is the degree of the
vertex $v$. From this data, we obtain a partition of
\begin{align*}
  \sum_{v \in V} (2g(v) - 3 + n(v)) &= 2g - 3(2g - 2 - d) + 2(2g - 3 - d) \\
  &= d
\end{align*}
by collecting the socle dimensions $2g(v) - 3 + n(v)$ for each vertex
$v \in V$ and throwing away the zeroes. We will call this partition
the \emph{housing data} of the pure boundary stratum. From the
$\lambda_g$ formula, it is easy to see that the pairing of the
$\kappa$-ring with a pure boundary stratum is determined by its
housing data.

On the other hand, it is interesting to consider which partitions of
$d$ can arise as housing data corresponding to a pure boundary
stratum. We will call these partitions \emph{housing partitions}.
\begin{lem}\label{lem:housepart}
  A partition $\sigma$ of $d$ is a housing partition if and only if it
  either has fewer than $2g - 2 - d$ parts or exactly $2g - 2 - d$
  parts, at least two of which are even.
\end{lem}
\begin{proof}
  Only partitions of length at most $2g - 2 - d$ can be housing
  partitions because there are only that many vertices. Furthermore, it
  is easy to see that no partition of $2g - 2 - d$ parts with fewer
  than two even parts can arise since every vertex with only one edge
  gives an even part (or no part if $g(v) = 1$).

  Now suppose $\sigma$ is a partition of $d$ with either fewer than
  $2g-2-d$ parts or exactly $2g-2-d$ parts with at least two even. Let
  $(\tau_i)_{1\le i\le 2g-2-d}$ be the tuple of nonnegative integers
  given by appending $2g-2-d-\ell(\sigma)$ zeroes to $\sigma$, so that
  the sum of the $\tau_i$ is $d$, and exactly $2k+2$ of the $\tau_i$
  are even for some nonnegative integer $k$.

  Construct a tree $\Gamma$ by taking a path of $2g-2-d-k$ vertices
  and adding $k$ additional leaves connected to vertices
  $2,3,\ldots,k+1$ along the path respectively. Thus, $\Gamma$ has
  $2g-2-d$ vertices, each of degree at most three, and exactly $2k+2$
  of the vertices of $\Gamma$ have odd degree. We now choose a
  bijection between the $\tau_i$ and the vertices of $\Gamma$ such
  that even $\tau_i$ are assigned to vertices of odd degree. We can
  then assign a genus $g_i = (\tau_i+3-n_i)/2$ to each vertex, where
  $n_i$ is the degree of the vertex to which $\tau_i$ was
  assigned. The resulting stable tree has housing data $\sigma$, as
  desired.
  \hfill $\Box$
\end{proof}

\subsection{Reduction to a combinatorial problem}\label{sec:reduction}

We have already described the housing data of a pure boundary
stratum. Let us now describe a similar notion for any class in the
generating set described in Section~\ref{sec:generators}.  Such a
class is given by a boundary stratum corresponding to a tree $\Gamma =
(V, E)$ and a genus assignment $g\colon V \to \mathbb Z_{\ge 0}$, along
with assignments of monomials in $\kappa$- and $\psi$-classes (of
degrees $r(v)$ and $s(v)$ respectively) to each component of the
stratum. Let $k = \sum_{v \in V} (r(v) + s(v))$; then we must have
$|E| = 2g - 3 - d - k$ edges in the tree in order to obtain a class of
degree $2g - 3 - d$. If this class does not vanish by dimension
reasons, then we can obtain a partition $\gamma$ of
\begin{equation*}
  \sum_{v \in V} (2g(v) - 3 + n(v) - r(v) - s(v)) = 2g - 3(2g - 2 - d - k) + 2(2g - 3 - d - k) - k
= d
\end{equation*}
by assigning to each vertex of $V$ the number $2g(v) - 3 + n(v) - r(v)
- s(v)$. This is exactly the degree $d'(v)$ such that the
$\lambda_{g(v)}$-pairing of $R^{d'(v)}(M_{g(v), n(v)}^c)$ with the
monomial of $\psi$- and $\kappa$-classes at $v$ is not zero for
dimension reasons. Then, the pairing with the boundary class is
determined by the partition $\gamma$, an assignment of degrees $r(i)$
and $s(i)$ to the parts $i \in I(\gamma)$ and partitions $\tau_i \in
P(r(i))$ and $\rho_i \in P(s(i))$ corresponding to the $\kappa$- and
$\psi$-monomials. In particular we can leave out classes which were
assigned to vertices with $2g(v) - 3 + n(v) - r(v) - s(v) = 0$ and we
do not need to remember which node corresponds to each $\psi$. The
result of the $\lambda_g$-pairing of this class together with a
$\kappa$-monomial corresponding to a partition $\pi$ of $d$ is (up to
scaling) given by
\begin{equation*}
  \sum_\varphi \prod_{j \in I(\gamma)} \vartheta\left(\pi_{\varphi^{-1}(j)}, \tau_j; \rho_j\right),
\end{equation*}
where the sum runs over all refining functions $\varphi$ of $\gamma$
into $\pi$.

When we view $\mathbb Q^{P(d)}$ as a ring of formal
$\kappa$-polynomials, this pairing gives linear forms
$$v_{\gamma, \{\tau_i\}, \{\rho_i\}} \in \left(\mathbb
  Q^{P(d)}\right)^*.$$
We notice that the formulas still make combinatorial sense even if the
triple $(\gamma, \{\tau_i\}, \{\rho_i\})$ does not come from pairing
with an actual tautological class.

The special case where all the $r(i)$ and $s(i)$ are zero gives the pairing of $\kappa$
classes with pure boundary classes. We get $|P(d)|$ linear forms
$M_\lambda$, which we normalize such that $M_\lambda(\lambda) = 1$:
\begin{equation}\label{eq:M_lambda}
  M_\lambda(\pi) = \frac 1 {\mathrm{Aut}(\lambda)} \sum_\varphi \prod_{j \in I(\lambda)}\vartheta \left(\pi_{\varphi^{-1}(j)}\right).
\end{equation}
In this way we obtain a basis of $\left(\mathbb Q^{P(d)}\right)^*$. If
we sort partitions in any way such that shorter partitions come before
longer partitions, then the basis change matrix from this basis to the
standard basis is triangular with ones on the diagonal. Note that this
basis uses some partitions which are not housing partitions.

The housing theorem can now be reformulated as follows:

\bigskip
\noindent
{\bf Claim.}
  {\em The span of $\{M_\lambda : \lambda\text{ is a housing partition}\}$
  in $\left(\mathbb Q^{P(d)}\right)^*$ equals the span of the
  $v_{\gamma, \{\tau_i\}, \{\rho_i\}}$ for all choices of housing
  data.}

\bigskip

To prove this claim, we will first in Section~\ref{sec:matrix} express
the vectors $v_{\gamma,
  \{\tau_i\}, \{\rho_i\}}$ for any choice of housing data in terms of
the basis of $\left(\mathbb Q^{P(d)}\right)^*$ we have described above.
We will then in Section~\ref{sec:reinter}
rewrite the coefficients as counts of certain combinatorial
objects. This combinatorial interpretation is proved in
Section~\ref{sec:mainclm}. We conclude in Section~\ref{sec:second} by
showing that when expressing vectors $v$ corresponding to actual
housing data in terms of the $M_\lambda$, the coefficient is zero
whenever $\lambda$ is not a housing partition.

\subsection{A Matrix Inversion}\label{sec:matrix}

In Section~\ref{sec:reduction} we have seen that there are formal
expansions
\begin{equation*}
  v_{\gamma, \{\tau_i\}, \{\rho_i\}} = \sum_{\lambda \in P(d)} c_{\lambda, \gamma, \{\tau_i\}, \{\rho_i\}} M_\lambda
\end{equation*}
for some coefficients $c_{\lambda, \gamma, \{\tau_i\},
  \{\rho_i\}}$.

We can calculate $c_{\lambda, \gamma, \{\tau_i\}, \{\rho_i\}}$
explicitly by inverting the triangular matrix given by equation
\eqref{eq:M_lambda}. We obtain
\begin{align*}
  c_{\lambda, \gamma, \{\tau_i\}, \{\rho_i\}} = \sum_{l = 0}^\infty
  (-1)^l \sum_ {\lambda_0 \stackrel{\varphi_1}{\to} \dots
    \stackrel{\varphi_l}{\to} \lambda_l \stackrel{\varphi_{l +
        1}}{\to} \gamma} \frac{v_{\gamma, \{\tau_i\}, \{\rho_i\}}
    (\lambda_l)}{\prod_{i = 1}^l |\mathrm{Aut}(\lambda_i)|}
  \prod_{i = 1}^l \prod_{j \in I(\lambda_i)} \vartheta\left((\lambda_{i -
      1})_{\varphi_i^{-1}(j)}\right),
\end{align*}
where we sum over chains $\lambda = \lambda_0, \dotsc \lambda_l$ of
refinements of $\gamma$ with corresponding refining functions
$\varphi_i$. In particular, $c_{\lambda, \gamma, \{\tau_i\},
  \{\rho_i\}} = 0$ if $\lambda$ is not a refinement of $\gamma$.

We can reduce to the special case in which $\gamma = (d)$ is of length
one by splitting this sum based on the composition $\varphi :=
\varphi_{l + 1} \circ \varphi_l \circ \dots \circ \varphi_1$ and
examining the contribution of the preimages of the $j \in
I(\gamma)$. The result is
\begin{equation}\label{eq:general}
    c_{\lambda, \gamma, \{\tau_i\}, \{\rho_i\}} = \sum_\varphi \prod_{j
      \in I(\gamma)} c_{\lambda_{\varphi^{-1}(j)}, (\gamma_j), \{\tau_j\},
      \{\rho_j\}},
\end{equation}
summed over refining functions $\varphi$ of $\gamma$ into $\lambda$.

When $\gamma = (d)$, we set $\tau_1 =: \tau$ and $\rho_1 =: \rho$, and
we can write more compactly:
\begin{equation} \label{eq:sum}
  c_{\lambda,(d),\{\tau\}, \{\rho\}} = \sum_{l = 0}^\infty (-1)^l \sum_{\lambda_0 \stackrel{\varphi_1}{\to} \dots \stackrel{\varphi_l}{\to} \lambda_l} \frac{\vartheta(\lambda_l, \tau; \rho)}{\prod_{i = 1}^l |\mathrm{Aut}(\lambda_i)|} \prod_{i = 1}^l \prod_{j \in I(\lambda_i)} \vartheta\left((\lambda_{i - 1})_{\varphi_i^{-1}(j)}\right)
\end{equation}

\subsection{Interpreting the coefficients combinatorially}\label{sec:reinter}

We will interpret the coefficients $c_{\lambda, (d), \{\tau\},
  \{\rho\}}$ as counting certain permutations of symbols labeled by
the parts of the partitions $\lambda,\tau$, and $\rho$. We say that a
symbol is \emph{of kind $i$} if it is labelled by some $i$ belonging
to the disjoint union of the indexing sets of the partitions,
$I(\lambda)\sqcup I(\tau)\sqcup I(\rho)$. There will in general be
multiple symbols of a given kind.

\bigskip

\noindent
{\bf Main Claim.}
{\em  The coefficient $c_{\lambda, (d), \{\tau\}, \{\rho\}}$ counts the
  number of permutations of
  \begin{itemize}
    \item $\lambda_i + 1$ symbols of kind $i$ for each $i \in I(\lambda)$,
    \item $\tau_i + 1$ symbols of kind $i$ for each $i \in I(\tau)$, and
    \item $\rho_i$ symbols of kind $i$ for each $i \in I(\rho)$
  \end{itemize}
  such that:
  \begin{enumerate}
  \item[\rm (1)] \phantomsection\label{cond:11} If the last symbol of some kind $i$ is
    immediately followed by the first symbol of kind $j$ with $i, j
    \in I(\lambda)\sqcup I(\tau)$, then we have $i < j$.
  \item[\rm (2)] \phantomsection\label{cond:12} For $i \in I(\lambda)$, the last symbol of
    kind $i$ is not immediately followed by a symbol of kind $j$ for
    any $j \in I(\lambda)$,
  \end{enumerate}
  averaged over all total orders $<$ of $I(\lambda)\sqcup I(\tau)$
  such that elements of $I(\tau)$ are smaller than elements of
  $I(\lambda)$.
}

\bigskip

It follows in particular that the coefficient $c_{\lambda, \gamma,
  \{\tau_i\}, \{\rho_i\}}$ is non-negative.

\subsection{Proof of the Main Claim}\label{sec:mainclm}

\subsubsection{Refinements of permutations of symbols}\label{sec:symperm}

For given natural numbers $d$, $n$ and a partition $\tau \in P(d)$, we
will study permutations $S$ of $\tau_i + 1$ symbols of kind $i$ for
$i \in I(\tau)$ and $n$ symbols of kind $c$.
(The permutations of symbols appearing in the previous section are an
instance of this.)
We will need to construct refined permutations of this type for
partition refinements $\varphi\colon I(\sigma) \to I(\tau)$.
For this, we need additional \emph{refinement data}: for each
$i \in I(\tau)$, let $T_i$ be a permutation of $\sigma_j + 1$ symbols
of kind $j$ for $j \in \varphi^{-1}(i)$.

Given $S$ and the refinement data, we can obtain a permutation $S'$ of
$\sigma_i + 1$ symbols of kind $i$ and $n$ symbols of kind $c$ in the
following way: For each $i \in I(\tau)$ and each
$j \in \varphi^{-1}(i)$, modify $T_i$ by gluing the last symbol of
kind $j$ with the immediately following symbol; the result is a
permutation $T'_i$ of $\tau_i + 1$ symbols. To construct $S'$ from
$S$, for each $i$ we replace the symbols of kind $i$ by $T'_i$ and
then remove the glue.

\subsubsection{Reinterpretation}\label{sec:reinter2}

We start with a combinatorial interpretation of the number
$\vartheta(\sigma; \tau)$ for partitions $\sigma$ and $\tau$.

\begin{lem} \label{lem:tool}
  Given an arbitrary total order $<$ on $I(\sigma)$,\,the number $\vartheta(\sigma; \tau)$ is equal to the number of permutations\,of
  \begin{itemize}
  \item $\sigma_i + 1$ symbols of kind $i$ for each $i \in I(\sigma)$
    and
  \item $\tau_i$ symbols of kind $i$ for each $i \in I(\tau)$
  \end{itemize}
  such that the following property holds:

  If the last symbol of kind $i$ is immediately followed by the first
  symbol of kind $j$ for $i, j \in I(\sigma)$, then we have $i < j$.
\end{lem}
\begin{proof}
  For each permutation $S$ of symbols as above, but not necessarily
  satifying the property, we can assign a set partition $Q_S \vdash
  I(\sigma)$ which measures in what ways it fails to satisfy the
  property: $Q_S$ is the finest set partition such that if $i > j$ and
  the last symbol of kind $i$ is immediately followed by the first
  symbol of kind $j$ in $S$, then $i$ and $j$ are in the same part of
  $Q_S$. Thus $S$ satisfies the given property if and only if $Q_S$ is
  the set partition with all parts of size $1$.

  The multinomial coefficient in the summand in the formula for
  $\vartheta(\sigma; \tau)$ given by Lemma~\ref{lem:kappapsi}
  corresponding to a set partition $P \vdash I(\sigma)$ counts the
  number of permutations $S$ such that for
  $p = \{p_1,\ldots,p_k\} \in P$ with $p_1 > \cdots > p_k$, the last
  symbol of kind $p_i$ is immediately followed by the first symbol of
  kind $p_{i+1}$ in $S$ for $i = 1,\cdots,k-1$.
  These are precisely the $S$ such that $Q_S$ can be obtained by
  combining parts of $P$ such that the largest element in one part is
  smaller than the smallest element of the other part.

  This means that if we split the formula for
  $\vartheta(\sigma; \tau)$ given by Lemma~\ref{lem:kappapsi} into a
  sum over permutations $S$, the contribution of a permutation with
  failure set partition $Q = \{Q_1,\ldots,Q_k\}$ is precisely
  \begin{equation*}
    \prod_{i=1}^k\sum_{j = 0}^{|Q_k|-1}(-1)^j\binom{|Q_k|-1}{j},
  \end{equation*}
  which is $1$ for $Q$ the set partition with all parts of size $1$,
  and $0$ otherwise.
  \hfill $\Box$
\end{proof}

Equipped with Lemma~\ref{lem:tool}, the next step in the proof of the
Main Claim is to split the coefficient
$c_{\lambda, (d), \{\tau\}, \{\rho\}}$ into a sum over the set
$S_{\lambda,\tau,\rho}$ of permutations of $\lambda_i + 1$
(respectively, $\tau_i + 1$, $\rho_i$) symbols of kind $i$ for
$i \in I(\lambda)$ (respectively, $i \in I(\tau)$, $i \in I(\rho)$).
For this, we introduce the notion of the \emph{composite permutation}.

Consider the following data:
\begin{itemize}
\item a chain of partitions $\lambda_0 = \lambda, \lambda_1, \dots,
  \lambda_l$ with refining maps $\varphi_i$ as in \eqref{eq:sum},
\item the additional data of an order $<$ on
  $I(\lambda_l)\sqcup I(\tau)$ such that elements of $I(\tau)$ appear
  before elements of $I(\lambda_l)$,
\item the additional data of orders on $\varphi_i^{-1}(j)$ for
  $1 \le i \le l$ and $j \in I(\lambda_i)$.
\end{itemize}

With this data, we identify each $\kappa$ socle evaluation factor
\begin{equation*}
  \vartheta\left((\lambda_{i - 1})_{\varphi_i^{-1}(j)}\right)
\end{equation*}
with the number of permutations of $(\lambda_{i - 1})_k + 1$ symbols
of kind $k \in \varphi_i^{-1}(j)$ such that if the last symbol of kind
$k$ is immediately followed by the first symbol of kind $k'$, then
$k < k'$.
We can interpret each such permutation as refinement data
corresponding to the refining function $\varphi_i$ of $\lambda_{i+1}$ into
$\lambda_{i}$.

Furthermore, we interpret the factor
\begin{equation*}
  \vartheta\left(\lambda_l, \tau; \rho\right)
\end{equation*}
as the number of permutations $S_l$ of $(\lambda_l)_k + 1$,
$\tau_k + 1$ and $\rho_k$ symbols of kind $k$ with
$k \in I(\lambda_l)$, $k \in I(\tau)$ and $k \in I(\rho)$,
respectively, such that if the last symbol of kind $k$ is immediately
followed by the first symbol of kind $k'$ for
$k, k' \in I(\lambda_l)\sqcup I(\tau)$, then $k < k'$.

Given all this data, we can build the \emph{composite permutation} by
repeatedly refining the collection of symbols of kind $k$ with
$k \in I(\lambda_l)$ using the construction from
Section~\ref{sec:symperm} and keeping the order of the other symbols
intact.
The result is a permutation of $\lambda_k + 1$, $\tau_k + 1$ and
$\rho_k$ symbols of kind $k$ for $k \in I(\lambda)$, $k \in I(\tau)$
and $k \in I(\rho)$ respectively.

Using the combinatorial interpretations of
$\vartheta\left((\lambda_{i - 1})_{\varphi_i^{-1}(j)}\right)$ and
$\vartheta\left(\lambda_l, \tau; \rho\right)$, and the notion of the
composite permutation, we may therefore write
$c_{\lambda, (d), \{\tau\}, \{\rho\}}$ as a sum over the set
$S_{\lambda,\tau,\rho}$.
To remove the dependence on the chosen orders, we will average over
all choices of them.

We note that any composite permutation has the property that the last
symbol of any kind $j \in I(\lambda)$ is not immediately followed by
the first symbol of some kind $j' \in I(\tau)$.
Also, with the natural induced ordering, any composite partition
satisfies condition (\hyperref[cond:11]{1}) in the Main Claim.

\subsubsection{Simplification}

For any permutation in $S_{\lambda,\tau,\rho}$, we assign a set
partition $P \vdash I(\lambda)$, which measures in what way it fails
to satisfy condition (\hyperref[cond:12]{2}) in the Main Claim. We define $P$
to be the finest set partition such that if the last symbol of kind
$i$ is immediately followed by a symbol of kind $j$ for $i, j \in
I(\lambda)$ then $i$ and $j$ lie in the same set of $P$.

Now, suppose we are given a chain of partitions
$\lambda, \lambda_1, \dots, \lambda_l$ along with additional refining
data, orderings and base permutation $S_l$ as above. Let $P$ be the failure set partition of the composite permutation. Let $\lambda_{l + 1} := \lambda^P$ be the partition formed by merging parts of $\lambda$ according to $P$, so there is a canonical refining function $I(\lambda) \to I(\lambda^P)$. By the construction of the composite permutation, this function actually factors through a refining function
$\varphi'\colon I(\lambda_l) \to I(\lambda_{l + 1})$.
Suppose that this refining function $\varphi'$ is nontrivial, i.e.
$\lambda_l \ne \lambda^P$.

We note that if we change the order on $I(\lambda_l)\sqcup I(\tau)$
such that the order on $I(\tau)$ and each inverse image of $\varphi'$
is preserved, the failure partition $P$ does not change.
We will therefore group these orderings together.

On the other hand, consider the following data:
\begin{itemize}
\item the chain
  $\lambda, \lambda_1, \dots, \lambda_l, \lambda_{l + 1}$ with
  refining maps $\varphi_i$ and $\varphi'$ as before,
\smallskip
\item the orders and refining data corresponding to the $\varphi_i$ as
  before,
\item in addition, an order on each preimage of $\varphi'$ which is
  induced by the order on $I(\lambda_l) \sqcup I(\tau)$,
\item refining data corresponding to $\varphi'$ induced from the
  permutation corresponding to $\lambda_l$, $\tau$ and $\rho$,
\item any order on $I(\lambda_{l + 1}) \sqcup I(\tau)$ such that the
  restriction to $I(\tau)$ is the restriction of the order on
  $I(\lambda_l) \sqcup I(\tau)$ and such that elements of $I(\tau)$
  appear before elements of $I(\lambda_{l + 1})$,
\smallskip
\item permutations of $(\lambda_{l + 1})_i + 1$, $\tau_i + 1$,
  $\rho_i$ symbols of kind $i$ for $i \in I(\lambda_{l + 1})$,
  $i \in I(\tau)$ and $i \in I(\rho)$ respectively, defined from the
  permutation corresponding to $\lambda_l$ by leaving out the last
  symbol of any kind $i \in I(\lambda_l)$ which is not the last one in
  a level set of $\varphi'$ and identifying symbols according to
  $\varphi'$.
\end{itemize}

It is easy to check that the refining data and the permutation still
satisfy the order conditions.
Furthermore, the failure set partition of the composite partition of
this new data is still $P$, so that now $\lambda_{l + 1} = \lambda^P$.

The original chain with additional data giving failure set partition
$P$ and the extended chain with additional data contribute to
$c_{\lambda,(d),\{\tau\},\{\rho\}}$ in formula \eqref{eq:sum} with
opposite signs since the extended chain is one element longer. We
claim these contributions cancel.

For the original chain, we have
\begin{equation*}
 \frac{(\ell(\lambda_l))!}{\prod_{j \in I(\lambda_{l + 1})}|\varphi'^{-1}(j)|!}
\end{equation*}
choices of orders on $I(\lambda_l) \sqcup I(\tau)$ in the above
construction. For the extended chain, we made
\begin{equation*}
  |\mathrm{Aut}(\lambda_{l + 1})| (\ell(\lambda_{l + 1}))!
\end{equation*}
choices in the above construction.

However, the contributions are also weighted by averaging over choices
of orders and by the automorphism factors in \eqref{eq:sum}.
For the original chain, the weight is
\begin{equation*}
  ((\ell(\lambda_l))!)^{-1},
\end{equation*}
and for the extended chain, the weight is
\begin{equation*}
  \left(|\mathrm{Aut}(\lambda_{l + 1})| (\ell(\lambda_{l + 1}))! \prod_{j \in I(\lambda_{l + 1})}|\varphi'^{-1}(j)|!\right)^{-1}.
\end{equation*}
Thus, the two contributions cancel.

The only remaining contributions occur when $l = 0$ and $P$ is the set
partition into one-element sets.
Since they are weighted with coefficient one, this finishes the proof
of the Main Claim.

\subsection{Proof of the Housing Theorem}\label{sec:second}

We begin with a simple lemma.
\begin{lem}\label{lem:van}
  Suppose $r + s + \ell(\gamma) < \ell(\lambda)$. Then $c_{\lambda,
    \gamma, \{\tau_i\}, \{\rho_i\}} = 0$.
\end{lem}
\begin{proof}
  We examine the summand in formula \eqref{eq:general} corresponding
  to some $\varphi$. A factor in this summand can only be nonzero if
  $r(i) + s(i) + 1\ge \ell(\varphi^{-1}(i))$. Therefore each summand
  will vanish unless $r + s + \ell(\gamma) \ge \ell(\lambda)$.~\hfill $\Box$
\end{proof}

Now let us suppose that $(\gamma$, $\{\tau_i\}$, $\{\rho_i\})$ is the
housing data of a boundary class of the generating set. We need to
show that $c_{\lambda, \gamma, \{\tau_i\}, \{\rho_i\}} = 0$ for each
$\lambda$ which is not a housing partition.

Let us first study the case $\ell(\lambda) > 2g - 2 - d$. Since
$\gamma$ is derived from a boundary stratum of codimension at most
$2g-3-d-r-s$ (we are missing the $\psi$- and $\kappa$-classes from the
components, which do not contribute to $\gamma$) by diminishing parts
by their $\kappa$- and $\psi$-degrees, we have the inequality
$\ell(\gamma) \le 2g - 2 - d - r - s$. Then by Lemma~\ref{lem:van} we
are done in this case. The same argument settles also the case where
there are components of the boundary stratum we are considering which
do not appear in $\gamma$ and $\ell(\lambda) = 2g - 2 - d$.

Now assume that $\ell(\lambda) = 2g - 2 - d$ and that $\lambda$
contains no even part. Then by the same arguments if the coefficient
is nonzero, we must have $\ell(\gamma) = 2g - 2 - d - r - s$. Furthermore,
from the proof of Lemma~\ref{lem:van} we see that $r(i) + s(i) =
\ell(\varphi^{-1}(i)) - 1$ for each $i \in I(\gamma)$. This implies
$\ell(\varphi^{-1}(i)) + r(i) + s(i) \equiv 1 \pmod{2}$, and therefore
for each $i \in I(\gamma)$, we have $\gamma_i + r(i) + s(i) \equiv 1
\pmod{2}$. Hence each part of the housing data (for the underlying
boundary stratum), which $\gamma$ was obtained from by subtraction of
$r(i) + s(i)$ from each part, is odd. This is a contradiction, so the
coefficient must be zero, as desired.

\section{The Rank Theorem}\label{sec:rank}

\subsection{Reformulation}

Let us first formulate a stronger version of the Rank Theorem.
\begin{thm}\label{thm:rank}
  For any $\kappa$-polynomial $F$ in degree $r := 2g - 3 - d$, the
  following two statements are equivalent:
  \begin{enumerate}
  \item[\rm (1)] For any $\pi \in P(g - 2 - r)$, we have $\int_{\overline M_g}
    F \kappa_\pi \lambda_g\lambda_{g - 1} = 0$.
  \item[\rm (2)] There is a $B \in PBR^r(M_g^c)$ such that for any $\pi' \in
    P(2g - 3 - r)$ we have $\int_{\overline M_g} (F - B) \kappa_{\pi'}
    \lambda_g = 0$.
  \end{enumerate}
\end{thm}

It will be convenient to show that we can replace the first condition
in Theorem~\ref{thm:rank} by
\begin{enumerate}\setcounter{enumi}{2}
\item[\rm (3)] For any $\pi \in P(g - 2 - r)$ of length at most $r + 1$ we have
  $\int_{\overline M_g} F \kappa_\pi \lambda_g\lambda_{g - 1} = 0$.
\end{enumerate}
Then Theorem~\ref{thm:rank} will follow from the following simple
argument. Consider an $F$ satisfying the second condition and we want
to show that $\int_{\overline M_g} F \kappa_\pi \lambda_g\lambda_{g - 1} = 0$
for some given $\pi \in P(g - 2 - r)$. Notice that then also
$F\kappa_\pi$ satisfies the second condition since $B\kappa_\pi$ lies
in $BR^{g - 2}(M^c_g)$ and by the Housing Theorem can be replaced by
some $B' \in PBR^{g - 2}(M^c_g)$. We then find that $\int_{\overline M_g} F
\kappa_\pi \lambda_g\lambda_{g - 1} = 0$ since in this case the length
condition is trivial.

Let $v_F\in (\mathbb Q^{P(d)})^*$ be defined by the $\lambda_g$-pairing of
$F$ with the kappa ring. The linear forms $M_\lambda$ (defined in
Section~\ref{sec:reduction})
form a basis for $(\mathbb Q^{P(d)})^*$, so $v_F$ can be written as a linear
combination of them. The second condition in the theorem is then equivalent
to the condition that only the $M_\lambda$ coming from actual boundary strata
(those where $\lambda$ is a housing partition) are used in this expansion
of $F$. Notice that by Lemma~\ref{lem:van} we can assume that the $\lambda$
are partitions of $2g - 3 - r$ of length at most $r +
1$. By Lemma~\ref{lem:housepart}, this means that the only non-housing
partitions that might occur are
of length exactly $r + 1$ with only odd parts. For the proof of the
Rank Theorem we will
need to understand the coefficients corresponding to these partitions
better.

Observe that partitions of $2g - 3 - r$ of length exactly $r +
1$ with only odd parts correspond to partitions of $g - 2 - r$ of
length at most $r + 1$. So for any $\sigma \in P(g - 2 - r, r + 1)$ we
can look at $\eta_\sigma, \mu_\sigma \in (\mathbb Q^{P(r)})^*$ with
$\eta_\sigma(\tau) := c_{\lambda, (d), \{\tau\}, \{\emptyset\}}$, where $\lambda$
is the partition of $2g - 3 - r$ of length $r + 1$ corresponding to
$\sigma$, and $\mu_\sigma(\tau)$ is up to a factor the integral
$\int_{\overline M_g} \kappa_\sigma \kappa_\tau \lambda_g\lambda_{g - 1}$,
namely
\begin{equation*}
  \mu_\sigma(\tau) = \sum_{P \vdash I(\sigma) \sqcup I(\tau)} (-1)^{\ell(\sigma) + \ell(\tau) + |P|} \frac{(2g - 3 + |P|)!}{\prod_{i \in P}(2(\sigma, \tau)^P_i + 1)!!}.
\end{equation*}
So what we need to show is the following:

\bigskip

\noindent{\bf Claim.}
{\em
  The $\mathbb Q$-subspaces of $(\mathbb Q^{P(r)})^*$ spanned by
  $\eta_\sigma$ and $\mu_\sigma$ for $\sigma$ ranging over all
  partitions of $g - 2 - r$ of length at most $r + 1$ are equal.}

\bigskip

Recall from Section~\ref{sec:reinter2} that $\eta_\sigma(\tau)$ is the
number of all permutations $S$ of $\lambda_i + 1$ symbols of kind $i
\in I(\lambda)$ and $\tau_i + 1$ symbols of kind $i \in I(\tau)$
satisfying
\begin{enumerate}
\item[\rm (1)] \phantomsection\label{cond:21} The last symbol of kind $i$ for some $i \in
  I(\lambda)$ is either at the end of the sequence or immediately
  followed by a symbol of kind $j$ for some $j \in I(\tau)$ which is
  not the first of its kind.
\item[\rm (2)] \phantomsection\label{cond:22} The successor of the last element of kind $i$ is
  not the first element of kind $j$ for any $i, j \in I(\tau)$ with $i
  < j$, where we fix some order on $I(\tau)$.
\end{enumerate}

Before coming to the main part of the proof we apply an invertible
transformation $\Phi$ to $(\mathbb Q^{P(r)})^*$ to simplify the
definitions of $\eta$ and $\mu$. The inverse of the transformation we
want to apply sends a linear form $\varphi' \in (\mathbb Q^{P(r)})^*$
to the linear form $\varphi$ defined by
\begin{equation*}
  \varphi(\tau) = \sum_{P \vdash I(\tau)}(-1)^{\ell(\tau) + |P|} \varphi'(\tau^P).
\end{equation*}
The transformation $\Phi$ defined in this way is clearly
invertible. By a similar argument as in the proof of
Lemma~\ref{lem:tool}, we can show that the image $\eta'_\sigma$ of
$\eta_\sigma$ under $\Phi$ is defined in the same way as $\eta_\sigma$
but leaving out Condition~(\hyperref[cond:22]{2}) on the permutations.

To study the action of $\Phi$ on $\mu$, we use the following lemma:

\begin{lem}
  Let $F$ be a function $F\colon P(n + m) \to \mathbb Q$ and
  define for any $\sigma \in P(n)$ functions $G_\sigma, G'_\sigma\colon
  P(m) \to \mathbb Q$ in terms of $F$ by
  \begin{align*}
    G_\sigma(\tau) &= \sum_{P \vdash I(\sigma) \sqcup I(\tau)} F((\sigma\sqcup \tau)^P) \\
    G'_\sigma(\tau) &= \sum_{\substack{P \vdash I(\sigma) \sqcup I(\tau) \\ P\text{ separates }I(\tau)}} F((\sigma\sqcup \tau)^P),
  \end{align*}
  where the second sum just runs over set partitions $P$ such that
  each element of $I(\tau)$ belongs to a separate part. Then
  \begin{equation*}
    G_\sigma(\tau) = \sum_{P \vdash I(\tau)} G'_\sigma(\tau^P).
  \end{equation*}
\end{lem}
\begin{proof}
  Given set partitions $P$ of $I(\tau)$ and $Q$ of $I(\sigma)\sqcup
  I(\tau^P)$, with $Q$ separating $I(\tau^P)$, we can alter $Q$ by
  replacing each element of $I(\tau^P)$ by the elements in the
  corresponding part of $P$. Each set partition of $I(\sigma)\sqcup
  I(\tau)$ is obtained exactly once by this construction.
  \hfill $\Box$
\end{proof}

Using this lemma and keeping track of the sign factors, we have that $\mu'_\sigma(\tau)$ is
\begin{equation*}
  \mu'_\sigma(\tau) = \sum_{\substack{P \vdash I(\sigma) \sqcup I(\tau) \\ P\text{ separates }I(\tau)}} (-1)^{\ell(\sigma) + \ell(\tau) + |P|} \frac{(2g - 3 + |P|)!}{\prod_{i \in P}(2(\sigma\sqcup \tau)^P_i + 1)!!}.
\end{equation*}
We can use the lemma again with the roles of $\sigma$ and $\tau$
interchanged to replace the generators of the span of $\mu'_\sigma$ by
$\mu''_\sigma$ with
\begin{equation}\label{eq:mu-final}
  \mu''_\sigma(\tau) := \sum_{\substack{P \vdash I(\sigma) \sqcup I(\tau) \\ P\text{ separates }I(\tau) \\ P\text{ separates }I(\sigma)}} (-1)^{\ell(\sigma) + \ell(\tau) + |P|} \frac{(2g - 3 + |P|)!}{\prod_{i \in P}(2(\sigma\sqcup \tau)^P_i + 1)!!}.
\end{equation}

Therefore we have reduced the proof of the Rank Theorem to proving the
following claim.

\bigskip
\noindent
{\bf Claim.} {\em
  The $\mathbb Q$-subspaces of $(\mathbb Q^{P(r)})^*$ spanned by
  $\eta'_\sigma$ and $\mu''_\sigma$ for $\sigma$ ranging over all
  partitions of $g - 2 - r$ of length at most $r + 1$ are equal.}

\subsection{Further strategy of proof}

In order to prove the claim we will establish interpretations for
$\eta'_\sigma(\tau)$ and $\mu''_\sigma(\tau)$ as counts of symbols of
different kinds satisfying some ordering constraints. This enables us
to find nonzero constants $F(i)$ for each $i \in I(\sigma)$
independent of $\tau$ such that
\begin{equation*}
  \mu''_\sigma(\tau) = \sum_{P \vdash I(\sigma)} \prod_{i \in P} F(i) \frac{\eta'_{\sigma^P}(\tau)}{(r + 1 - |P|)!},
\end{equation*}
giving a triangular transformation.

For the interpretations the notion of a \emph{comb-like order} plays
an important role. We say that symbols $i_1 \dotsc i_{2m + 1}$ are in
comb-like order if we have the relations $i_1 < i_3 < \dots < i_{2m +
  1}$ and $i_{2j} < i_{2j + 1}$ for $j \in [m]$. This is illustrated
in Figure~\ref{fig:comb}.

\begin{figure}[!htbp]
\begin{center}  \begin{tikzpicture}
    \begin{scope}[scale=0.5]
    \coordinate [label=left:$i_1$] (i1) at (0,0);
    \coordinate [label=right:$i_2$] (i2) at (1,0);
    \coordinate [label=left:$i_3$] (i3) at (0.5,1);
    \coordinate [label=right:$i_4$] (i4) at (1.5,1);
    \coordinate [label=left:$i_5$] (i5) at (1,2);
    \coordinate [label=left:$i_{2m -1}$] (j3) at (1.5,3);
    \coordinate [label=right:$i_{2m}$] (j2) at (2.5,3);
    \coordinate [label=left:$i_{2m + 1}$] (j1) at (2,4);
    
    \draw (i1) -- (i5);
    \draw (i3) -- (i2);
    \draw (i5) -- (i4);
    \draw[style=dotted] (i5) -- (j3);
    \draw (j3) -- (j1);
    \draw (j1) -- (j2);
    \end{scope}
  \end{tikzpicture}
  \end{center}
  \caption{A comb-like order}

  \vspace{-4cm}\phantomsection\label{fig:comb}

\vspace{4cm}\end{figure}
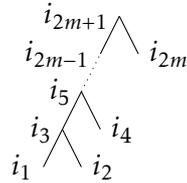

Note that the number of comb-like orderings of $2m + 1$ symbols is
$(2m + 1)! / (2m + 1)!!$. More generally the number
\begin{equation*}
  \frac{(2|\pi| + \ell(\pi))!}{\prod_{i \in I(\pi)}(2\pi_i + 1)!!}
\end{equation*}
corresponding to a partition $\pi$ counts the number of permutations
of the $2|\pi| + \ell(\pi)$ symbols $\bigcup_{i \in I(\pi)} \{i_1,
\dots,$ $i_{2\pi_i + 1}\}$ such that symbols corresponding to the same
part of $\pi$ appear in comb-like order.

\subsection{Combing orders}

We obtain a first reinterpretation of $\eta'_\sigma(\tau)$ by numbering
the symbols of the same kind:

\bigskip
\noindent
{\em Interpretation A1. }
  $\eta'_\sigma(\tau)$ is the number of all permutations of symbols
  $i_1, \dotsc, i_{\tau_i + 1}$ for $i \in I(\tau)$ and $i_1, \dotsc,
  i_{\lambda_i + 1}$ for $i \in I(\lambda)$ such that for fixed $i \in
  I(\tau) \sqcup I(\lambda)$ the $i_j$ appear in order and for all $i
  \in I(\lambda)$ the symbol $i_{\lambda_i + 1}$ is either at the end
  of the sequence or immediately followed by some $j_k$ for $j \in
  I(\tau)$ and $k \neq 1$.

\bigskip

Since $\lambda$ has length $r+1$ and $|\tau| = r$, such a permutation
gives a bijection between the $j_k$ for $j \in I(\tau)$ with $k \neq
1$ and all but one of the $i_{\lambda_i+1}$ for $i \in
I(\lambda)$. After picking this bijection, we can remove the
$i_{\lambda_i+1}$.

\bigskip
\noindent
{\em Interpretation A2. }
  $\eta'_\sigma(\tau)$ is the sum over bijections
  \begin{equation*}
    \varphi\colon I(\lambda) \to \{i_j \mid i \in I(\tau), j \neq 1\} \sqcup
    \{\text{End}\}    
  \end{equation*}
  of the number of permutations of symbols $i_1, \dotsc, i_{\tau_i +
    1}$ for $i \in I(\tau)$ and $i_1, \dotsc, i_{\lambda_i}$ for $i
  \in I(\lambda)$ such that symbols of the same kind appear in order
  and all symbols $i_j$ for $i \in I(\lambda)$ appear before
  $\varphi(i)$ (this condition is empty if $\varphi(i) = \text{End}$).

\bigskip

We can then add new symbols immediately following each $i_{\lambda_i}$
for $i\in I(\lambda)$ and reindex the $i_j$ for $i \in I(\tau)$ to
create comb-like orderings.

\bigskip
\noindent
{\em Interpretation A3. }\phantomsection\label{intermediateA}
  $\eta'_\sigma(\tau)$ is the sum over bijections
  \begin{equation*}
    \varphi\colon I(\lambda) \to \{i_j \mid i \in I(\tau), j\text{ even}\} \sqcup
    \{\text{End}\}     
  \end{equation*}
  of the number of permutations of symbols $i_1, \dotsc, i_{2\tau_i +
    1}$ for $i \in I(\tau)$, $i_1, \dotsc, i_{\lambda_i}$ for $i \in
  I(\lambda)$, and an additional symbol End such that the $i_j$ for $i
  \in I(\tau)$ appear in comb-like order, the $i_j$ for $i \in
  I(\lambda)$ appear in order, and $i_{\lambda_i}$ for $i \in
  I(\lambda)$ is immediately followed by $\varphi(i)$.

\bigskip

Recall that $\lambda$ is defined in terms of $\sigma$ by taking the
numbers $2\sigma_i + 1$ for each $i \in I(\sigma)$ and adding as many
ones as needed to reach length $r + 1$. There is only one symbol $i_1$
of kind $i$ for $i \in I(\lambda) \setminus I(\sigma)$ in
Interpretation~\hyperref[intermediateA]{A3} of $\eta'_\sigma(\tau)$ and it must
be immediately followed by $\varphi(i)$. For convenience set $(r + 1 -
\ell(\sigma))!  \cdot \eta''_\sigma := \eta'_\sigma$. Removing these
symbols $i_1$ gives an interpretation of $\eta''_\sigma$.

\bigskip
\noindent
{\em Interpretation A4. }\phantomsection\label{finalA}
  $\eta''_\sigma(\tau)$ is a sum over all injections
  \begin{equation*}
    \varphi\colon I(\sigma) \to \{i_j \mid i \in I(\tau), j\text{ even}\} \sqcup
    \{\text{End}\}
  \end{equation*}
  of the number of permutations of symbols $i_1, \dotsc, i_{2\tau_i +
    1}$ for $i \in I(\tau)$, $i_1, \dotsc, i_{2\sigma_i + 1}$ for $i
  \in I(\sigma)$, and an additional symbol End such that the $i_j$ for
  $i \in I(\tau)$ appear in comb-like order, the $i_j$ for $i \in
  I(\sigma)$ appear in order, and $i_{2\sigma_i + 1}$ for $i \in
  I(\sigma)$ is immediately followed by $\varphi(i)$.

\bigskip

We now switch to the interpretation of $\mu''_\sigma(\tau)$, which was
defined in \eqref{eq:mu-final}. The coefficient corresponding to a set
partition $P$ can be interpreted as the number of permutations of
symbols $i_1, \dotsc, i_{2(\sigma \sqcup \tau)^P_i + 1}$ for $i \in
I((\sigma\sqcup \tau)^P)$ and one additional symbol $\star$ such that
all $i_1, \dotsc, i_{2(\sigma\sqcup \tau)^P_i + 1}$ for $i \in
I((\sigma\sqcup \tau)^P)$ appear in comb-like order.

Because of the restrictions in the sum, the parts of $P$ are either
singletons or contain exactly one element from each of $I(\sigma)$ and
$I(\tau)$. This defines a function $\psi: I(\sigma) \to I(\tau) \sqcup
\{\star\}$, injective when restricted to the preimage of
$I(\tau)$. Interpreting the summands as counting comb-like orders and
cutting combs into two pieces for each part of $P$ of size two gives
the following:

\bigskip
\noindent
{\em Interpretation B1. }
  $\mu''_\sigma(\tau)$ is the sum over functions
  \begin{equation*}
    \psi\colon I(\sigma) \to I(\tau) \sqcup \{\star\}
  \end{equation*}
  such that $\psi|_{\psi^{-1}(I(\tau))}$ is injective, of a sign of
  $(-1)^{|\psi^{-1}(I(\tau))|}$ times the number of permutations of symbols
  $i_1, \dotsc, i_{2\tau_i + 1}$ for $i \in I(\tau)$, $i_1, \dotsc,
  i_{2\sigma_i + 1}$ for $i \in I(\sigma)$ and one additional symbol
  $\star$ such that all $i_1, \dotsc, i_{2\tau_i + 1}$ for $i \in
  I(\tau)$ and all $i_1, \dotsc, i_{2\sigma_i + 1}$ for $i \in
  I(\sigma)$ appear in comb-like order and such that $i_{2\sigma_i +
    1}$ for $i \in I(\sigma)$ with $\psi(i) \neq \star$ is immediately
  followed by $\psi(i)_1$.

\bigskip

Now we split the set of such permutations depending on the symbols
immediately following symbols $i_{2\sigma_i + 1}$ for $i \in
I(\sigma)$. We notice that the signed sum exactly kills those
permutations where some $i_{2\sigma_i + 1}$ for $i \in I(\sigma)$ is
immediately followed by some $j_1$ for $j \in I(\tau)$ since if such a
summand appears for some $\psi$ with $\psi(i) \neq j$ we must have
$\psi(i) = \star$ and we find the same summand with opposite sign in
the sum corresponding to the map $\psi'$ defined by $\psi'(i) = j$ and
$\psi'(k) = \psi(k)$ for $k \neq i$ and vice versa.

\bigskip
\noindent
{\em Interpretation B2. }\phantomsection\label{finalB}
  $\mu''_\sigma(\tau)$ is the number of permutations of symbols $i_1,
  \dotsc, i_{2\tau_i + 1}$ for $i \in I(\tau)$, $i_1, \dotsc,
  i_{2\sigma_i + 1}$ for $i \in I(\sigma)$ and one additional symbol
  $\star$ such that all $i_1, \dotsc, i_{2\tau_i + 1}$ for $i \in
  I(\tau)$ and all $i_1, \dotsc, i_{2\sigma_i + 1}$ for $i \in
  I(\sigma)$ appear in comb-like order and such that $i_{2\sigma_i +
    1}$ for $i \in I(\sigma)$ is not immediately followed by a symbol
  of the form $j_1$ with $j \in I(\tau)$.

\bigskip

Interpretations~\hyperref[finalA]{A4} and \hyperref[finalB]{B2} are very close. The
differences between the two of them are that the $\sigma$-type symbols
are in total order rather than comb-like order in \hyperref[finalA]{A4} and
that the conditions on the elements immediately following the
$i_{2\sigma_i + 1}$ are different.

We now break $\mu''_\sigma(\tau)$ into a sum over set partitions $P$
of $I(\sigma)$. Given a permutation of the symbols appearing in
Interpretation~\hyperref[finalB]{B2}, define a function
\[
\varphi: I(\sigma) \to \{i_j \mid i \in I(\tau),
j\text{ even}\} \sqcup \{\text{End}\}
\]
recursively by
\begin{equation*}
  \varphi(i) = \left\{
  \begin{aligned}
    j_{2k} && \text{if }i_{2\sigma_i + 1}\text{ for }i \in
    I(\sigma)\text{ is immediately
      followed by a symbol }\\
    && \text{ of the form }j_{2k}\text{ or }j_{2k + 1}\text{ with }j
    \in I(\tau), \\
    \text{End} && \text{if }i_{2\sigma_i + 1}\text{ for }i \in
    I(\sigma)\text{
      is immediately followed by }\star \\
    && \text{ or at the end of the sequence}, \\
    \varphi(j) && \text{if }i_{2\sigma_i + 1}\text{ for }i \in
    I(\sigma)\text{ is immediately
      followed by a symbol }\\
    && \text{ of the form }j_k\text{ with }j \in I(\sigma).
  \end{aligned}
  \right.
\end{equation*}

Then let $P$ be the set partition of preimages under $\varphi$.  We
will identify the summand of $\mu''_\sigma(\tau)$ corresponding to
such a set partition $P$ as $\eta''_{\sigma^P}(\tau)$ times a factor
depending only on $\sigma$ and $P$.

This factor is equal to
\[
\prod_{i\in P} F(i),
\]
where
\[
  F(i) = \frac{(2\sigma^P_i + |i| + 1)!}{\prod_{j \in i}(2\sigma_j + 1)!!}.
\]
Here, $F(i)$ should be interpreted as the number of permutations of
$2\sigma_j + 1$ symbols of kind $j$ for each $j \in i$ and one
additional symbol End such that the symbols of each kind appear in
comb-like order. If these permutations are interpreted as refinement
data, then the permutations counted by the $P$-summand of
$\mu''_\sigma(\tau)$ are the refinements by this data of the
permutations counted by $\eta''_{\sigma^P}(\tau)$.

Thus we have the identity
\begin{equation*}
  \mu''_\sigma = \sum_{P \vdash I(\sigma)} \prod_{i \in P} F(i) \eta''_{\sigma^P}.
\end{equation*}
This is a triangular change of basis with nonzero entries on the
diagonal, so the $\mu''$ and $\eta''$ span the same subspace in
$(\mathbb Q^{P(r)})^*$. This completes the proof of the Rank Theorem.

\providecommand{\bysame}{\leavevmode\hbox to3em{\hrulefill}\thinspace}
%
%

\bibliographystyle{amsalpha}
\bibliographymark{References}
\def\cprime{$'$}

\end{document}